\documentclass[10pt,journal]{IEEEtran}
\usepackage[ansinew]{inputenc}
\usepackage[dvips]{graphicx}
\usepackage[T1]{fontenc}
\usepackage{amsmath,enumerate,amssymb,hhline,verbatim}

\begin{document}

\title{The existence of fractional repetition codes}
\author{Toni Ernvall}

\maketitle

\newtheorem{definition}{Definition}[section]
\newtheorem{thm}{Theorem}[section]
\newtheorem{proposition}[thm]{Proposition}
\newtheorem{lemma}[thm]{Lemma}
\newtheorem{corollary}[thm]{Corollary}
\newtheorem{exam}{Example}[section]
\newtheorem{conj}{Conjecture}
\newtheorem{remark}{Remark}[section]

\newcommand{\La}{\mathbf{L}}
\newcommand{\h}{{\mathbf h}}
\newcommand{\Z}{{\mathbf Z}}
\newcommand{\R}{{\mathbf R}}
\newcommand{\C}{{\mathbf C}}
\newcommand{\D}{{\mathcal D}}
\newcommand{\F}{{\mathbf F}}
\newcommand{\HH}{{\mathbf H}}
\newcommand{\OO}{{\mathcal O}}
\newcommand{\G}{{\mathcal G}}
\newcommand{\A}{{\mathcal A}}
\newcommand{\B}{{\mathcal B}}
\newcommand{\I}{{\mathcal I}}
\newcommand{\E}{{\mathcal E}}
\newcommand{\PP}{{\mathcal P}}
\newcommand{\Q}{{\mathbf Q}}
\newcommand{\M}{{\mathcal M}}
\newcommand{\separ}{\,\vert\,}
\newcommand{\abs}[1]{\vert #1 \vert}

\begin{abstract}
Salim El Rouayheb and Kannan Ramchandran introduced the concept of fractional repetition (FR) code. In their article it remained unsolved when we can find such codes. Here we give an exact characterization of situations when it is possible to find an FR code.
\end{abstract}

\section{The existence of FR codes}

The definition of fractional repetition code, introduced in \cite{fr}, is the following.
\begin{definition}
\emph{A Fractional Repetition} (FR) code $\mathcal{C}$, with repetition degree $\rho$ for an $(n,k,d)$ DSS, is a collection $\mathcal{C}$ of $n$ subsets $V_1, V_2, \dots ,V_n$ of a set $\Omega=\{1,\dots,\theta\}$ and of cardinality $d$ each, satisfying the condition that each element of $\Omega$ belongs to exactly $\rho$ sets in the collection.
\end{definition}

It is easy to see that we have to have
$$
\theta \rho = nd
$$
and
$$
n \leq \binom{\theta}{d}.
$$

These conditions are also sufficient:

\begin{thm}
Let $n,d,\theta,\rho$ be positive integers and let also the two conditions above hold. Then we have an FR code $\mathcal{C}$, with repetition degree $\rho$ for an $(n,k,d)$ DSS.
\end{thm}
\begin{proof}
If $\theta=d$ then the claim is clear so assume that $\theta > d$.

Let $\omega$ be the smallest positive integer with property $d\omega \equiv 0 (\mod \theta)$. Write also $d\omega = \theta a$. Now we have $and=a\theta \rho= d\omega \rho$ and hence $\omega$ divides $an$. Since $\omega$ and $a$ are relatively prime we have that $\omega$ divides $n$.

Let $A,B \subseteq \Omega$. We say that $A$ is a cyclic shift of $B$ if we have $A=B+j$ for some integer $j$ where $B+j$ means that we add $j$ to each element of $B$ and then reduce the result modulo $\theta$. If $A$ is a cyclic shift of $B$ we write $A\sim B$. It is clear that $\sim$ is an equivalence relation. Denote the equivalence class of $A$ by $[A]$.

It is easy to check that
$$
[A] = \{ A, A+1, A+2, \dots, A+(\#[A]-1) \}
$$
and hence
$$
A = A + \#[A]
$$
for all sets $A \subseteq \Omega$. And since $A = A + \theta$ we have that $\#[A]$ divides $\theta$.

Define $\mathcal{D} = \{ A \subseteq \Omega | \#A=d \}$. Let $l \in \Omega$ and $A \in \mathcal{D}$. We have that $l$ belongs exactly to $d$ sets of sets $A , A+1 , \dots , A+(\theta-1)$. Here we might have that some of these sets are same and then we have calculated $l$ belonging to them that many times. This gives that $l$ belongs exactly to $\frac{d \#[A]}{\theta}$ of sets $[A]=\{ A, A+1, A+2, \dots, A+(\#[A]-1) \}$. Hence $\theta$ divides $d \#[A]$ and by the minimality of $\omega$ we know that $\omega$ divides $\#[A]$ for all sets $A \in \mathcal{D}$.

Define $S= \{ 1, 2, \dots, d \}$ and $\mathcal{T} = \{ S , S+d , S+2d , \dots , S + (\omega - 1)d \}$. It is clear that $\#[S]=\theta$, $\# \mathcal{T} = \omega$, and that each $l \in \Omega$ belongs exactly to $\frac{d \omega}{\theta}=a$ of sets in $\mathcal{T}$. Note that $\mathcal{T}+j$ has similar properties than $\mathcal{T}$ for all $j \in \Z$, and that
$$
[S] = \bigcup_{j \in \Z} (\mathcal{T}+j).
$$
Hence for some index set $\mathcal{I}$ with $\# \mathcal{I}=\frac{\theta}{\omega}$ we have that $\{ \mathcal{T} + j | j \in \mathcal{I} \}$ is a partition of $[S]$.

Now we are ready to construct the code $\mathcal{C}$. Write $n_0=n$.

First if $n_0 > \theta$ take any $A_{1} \in \mathcal{D}\setminus[S]$ and add $[A_{1}]$ into $\mathcal{C}$ (which is, at first, an empty set). Put $n_1=n_0 - \#[A_{1}]$.

Then if $n_1 > \theta$ take any $A_{2} \in \mathcal{D}\setminus([S]\cup[A_1])$ and add $[A_2]$ into $\mathcal{C}$. Put $n_2=n_1 - \#[A_2]$.

Continue this process as long as $n_j > \theta$. This is possible since we have originally $\binom{\theta}{d}-\#[S] \geq n-\theta$ sets available by our assumption. When we have $n_j \leq \theta$ we either have $n_j=0$ or $n_j > 0$ because $\#[A] \leq \theta$ for all sets $A \in \mathcal{D}$. If we have $n_j=0$ then we have $n$ sets in $\mathcal{C}$. If we have $n_j>0$ then we still know that $\omega$ divides $n_j$ since $\omega$ divides $n$ and $\omega$ divides $\#[A]$ for all sets $A \in \mathcal{D}$. Now $n_j=\omega e$ for some $0<e\leq\frac{\theta}{\omega}$. Take any $e$ of sets $\{ \mathcal{T} + j | j \in \mathcal{I} \}$ and add them into $\mathcal{C}$. Now we have $n$ sets in $\mathcal{C}$.

This $\mathcal{C}$ is a code with properties we wanted because its cardinality is $n$, all sets in it have cardinality $d$, and because each $l_1 \in \Omega$ belongs to as many of sets in $\mathcal{C}$ as any $l_2 \in \Omega$, we have that each $l \in \Omega$ belongs exactly to $\frac{nd}{\theta}=\rho$ of sets in $\mathcal{C}$.

\end{proof}

\begin{exam}
Let $\theta=8$, $d=6$, $n=12$, and $\rho=9$.

An algorithm to find a suitable FR code $\mathcal{C}$ when using the notation as in the proof of the theorem:

We have $n_0=12$, $\omega=4$, $S=\{ 1,2,3,4,5,6 \}$, and $\mathcal{T}= \{ S,S+6,S+4,S+2 \}$. Let $A_1= \{1,2,3,4,6,7\}$. Then $\#[A_1]=8$ and hence $n_1=12-8=4 \leq \theta$. So we find a suitable code by taking $\mathcal{C}=[A_1]\cup \mathcal{T}$.
\end{exam}

\begin{exam}
Let $\theta=7$, $d=3$, $n=21$, and $\rho=9$.

An algorithm to find a suitable FR code $\mathcal{C}$ when using the notation as in the proof of the theorem:

We have $n_0=21$, $\omega=7$, and $S=\{ 1,2,3 \}$. Let $A_1= \{ 1,2,4 \}$. Then $\#[A_1]=7$ and hence $n_1=21-7=14 > \theta$.

Let $A_2= \{ 1,3,5 \}$. Then $\#[A_2]=7$ and hence $n_2=14-7=7 = \theta$.

So we find a suitable code by taking $\mathcal{C}=[A_1]\cup [A_2]\cup [S]$.
\end{exam}

%\bibliographystyle{IEEE}	
%\bibliography{myrefs_special}	

\end{document}